\documentclass[leqno,reqno,tbtags,a4paper]{amsart}
\usepackage{amsthm,amsmath,amssymb}
\usepackage{stmaryrd,mathrsfs}
\usepackage{esint}
\usepackage{enumitem}
\usepackage{mathtools}

\allowdisplaybreaks

\newcommand{\ud}[0]{\,\mathrm{d}}
\newcommand{\dist}[0]{\operatorname{dist}}

\newcommand{\abs}[1]{|#1|}
\newcommand{\Babs}[1]{\Big|#1\Big|}
\newcommand{\Norm}[2]{\|#1\|_{#2}}
\newcommand{\BNorm}[2]{\Big\|#1\Big\|_{#2}}
\newcommand{\pair}[2]{\langle #1,#2 \rangle}
\newcommand{\Bpair}[2]{\Big\langle #1,#2 \Big\rangle}

\newcommand{\lspan}[0]{\operatorname{span}}
\newcommand{\bddlin}[0]{\mathscr{L}}
\newcommand{\supp}[0]{\operatorname{supp}}

\newcommand{\N}{\mathbb{N}}
\newcommand{\R}{\mathbb{R}}
\newcommand{\Z}{\mathbb{Z}}

\newcommand{\prob}[0]{\mathbb{P}}
\newcommand{\Exp}[0]{\mathbb{E}}

\newcommand{\good}[0]{\operatorname{good}}
\newcommand{\bad}[0]{\operatorname{bad}}

\swapnumbers \numberwithin{equation}{section}

\theoremstyle{plain}
\newtheorem{theorem}[equation]{Theorem}
\newtheorem{proposition}[equation]{Proposition}
\newtheorem{corollary}[equation]{Corollary}
\newtheorem{lemma}[equation]{Lemma}

\newtheorem{remark}[equation]{Remark}
\newtheorem{definition}[equation]{Definition}

\makeatletter
\@namedef{subjclassname@2020}{%
  \textup{2020} Mathematics Subject Classification}
\makeatother

\begin{document}

\title[Dyadic representation theorem using wavelets with compact support]{The dyadic representation theorem using smooth wavelets with compact support}
\author[T. \ Hyt\"onen and S. Lappas]{Tuomas \ Hyt\"onen and Stefanos Lappas}
\address{Department of Mathematics and Statistics, P.O.B.~68 (Pietari Kalmin katu~5), FI-00014 University of Helsinki, Finland}
\email{tuomas.hytonen@helsinki.fi, https://orcid.org/0000-0003-2911-2391}
\email{stefanos.lappas@helsinki.fi, https://orcid.org/0000-0001-5943-5672}

% 42B20 Singular and oscillatory integrals (Calder\'on-Zygmund, etc.) 
% 42C40 Wavelets and other special systems

\begin{abstract}
The representation of a general Calder\'on--Zygmund operator in terms of dyadic Haar shift operators first appeared as a tool to prove the $A_2$ theorem, and it has found a number of other applications. In this paper we prove a new dyadic representation theorem by using smooth compactly supported wavelets in place of Haar functions. A key advantage of this is that we achieve a faster decay of the expansion when the kernel of the general Calder\'on--Zygmund operator has additional smoothness.

\end{abstract}

\maketitle

\section{Introduction}
It was long conjectured that classical inequalities for singular integrals $T$ on weighted spaces $L^2(w)$ with a Muckenhoupt $A_2$ weight $w$ should take the sharp form
\begin{equation*}
  \Norm{Tf}{L^2(w)}\leq c_T[w]_{A_2}\Norm{f}{L^2(w)}.
\end{equation*}
This {\em $A_2$ conjecture} was first verified by one of us \cite{Hytonen:A2} by introducing a {\em dyadic representation} of $T$, an expansion in terms of simpler discrete model operators (called dyadic/Haar shifts). Earlier versions of the $A_2$ conjecture for special operators such as the martingale transform, the Beurling--Ahlfors transform, the Hilbert transform and the Riesz transform were due to J. Wittwer \cite{JW}, S. Petermichl and A.Volberg \cite{PV}, S. Petermichl \cite{Petermichl:Hilbert, Petermichl:Riesz}, respectively. Since then, simpler proofs of the $A_2$ theorem as in Lerner \cite{Lerner:simple}, Lacey \cite{L}, and Lerner--Ombrosi \cite{LO} replaced the dyadic representation by sparse domination, but the original dyadic representation theorem continues to have an independent interest and other applications.

One such application is the extension of the linear dyadic representation to bi-parameter (also known as product-space, or Journ\'e--type, after \cite{J}) singular integrals in \cite{Martikainen:Advances}, which has defined the new standard framework for the study of these operators. The multi-parameter extension of this is due to Y. Ou \cite{OU} and a bi-linear version is due to Li--Martikainen--Ou--Vuorinen \cite{LMOV}. In the bi-parameter context the representation theorem has proven to be extremely useful e.g., in connection with bi-parameter commutators and weighted analysis, see Holmes--Petermichl--Wick \cite{HPW}, Ou--Petermichl--Strouse \cite{OPS} and Li--Martikainen--Vuorinen \cite{LMV1,LMV2}. On the other hand, there are some fundamental obstacles to sparse domination of bi-parameter objects, see \cite{BCOR}, which makes the dyadic representation particularly useful in this setting.

In another direction, an open problem in vector-valued Harmonic Analysis is to describe the linear dependence of the norm of a vector-valued Calder\'on--Zygmund operator on the UMD constant of the underlying Banach space. In abstract UMD spaces, the linear bound has only been shown for the Beurling--Ahlfors transform and for some other special operators with even kernel such as certain {\em Fourier multiplier operators} (see \cite{MSSG}). It is also interesting to mention that, as was the case with the $A_2$ theorem, the linear bound for the Beurling--Ahlfors transform has been known for some time, yet the possible linear dependence between the vector-valued Hilbert transform and the UMD constant is still a famous open problem (see \cite[Problem~O.6]{HNVW}). More recently, Pott and Stoica established in \cite{PS} the linear dependence of sufficiently smooth Banach space-valued even singular integrals on the UMD constant by showing such a linear estimate for symmetric dyadic shifts. Their estimate for dyadic shifts grows like $2^{\max(i,j)/2}$ in terms of the parameters $(i,j)$ of the shifts. As explained in their work, to have convergence, one needs a decay factor $2^{-s\max(i,j)}$, which is guaranteed by kernel smoothness $s>\frac{1}{2}$ and only in dimension $d=1$. It is interesting to notice that in most other applications of the dyadic representation theorem, notably to the weighted inequalities, the rate of convergence of the representation is irrelevant as long as it is exponential. Formally, the same argument should work in any dimension $d$ assuming smoothness of order $s>\frac{1}{2}d$, but the existing Haar dyadic representation can only ''see'' smoothness up to order $s\leq1$; thus $\frac{1}{2}d<s\leq1$ forces $d=1$. 

This motivated us to find a new version of the dyadic representation theorem with faster decay using smooth wavelets with compact support. 
Our main result is the following (see Section \ref{basics} for a precise definition of the wavelet shifts $S^{ij}_{\omega}$ and the required regularity of the wavelets):
\begin{theorem}\label{main theorem}
Let $s\in\Z_+$, and $T$ be a bounded Calder\'on--Zygmund operator in $L^2(\R^d)$ with a kernel satisfying $\abs{\partial^{\alpha}K(x,y)}\leq\Norm{K}{CZ_s}\abs{x-y}^{-d-\abs{\alpha}}$ for every $\abs{\alpha}\leq s$. In addition, suppose that $T,T^*:\mathcal{P}_s\rightarrow\mathcal{P}_s$, where $\mathcal{P}_s$ is the space of polynomials of degree less than $s$. Then for any given $\epsilon>0$, $T$ has an expansion, say for $f,g\in C^1_c(\R^d)$,
\begin{equation*}
  \pair{g}{Tf}=c\cdot\big(\Norm{K}{CZ_s}+\Norm{T}{L^2\to L^2}\big)\cdot\Exp_{\omega} \sum_{i,j=1}^{\infty}2^{-(s-\epsilon)\max(i,j)}\pair{g}{S^{ij}_{\omega}f},
\end{equation*}
where $c$ depends only on d, s and $\epsilon$, $\Exp_{\omega}$ is the expectation with respect to the random parameter $\omega$, and $S^{ij}_{\omega}$ is a version of a dyadic shift with parameters $(i,j)$ but using sufficiently regular wavelets in place of the Haar functions. 
\end{theorem}

Having this result at our disposal, we can hope to extend the result of \cite{PS} to dimensions $d>1$. We plan to address this question in future work. Another possible area of applications is numerical algorithms for singular integrals, as in \cite{BCR}, where an ancestor of the dyadic representation is used for this purpose. It is clear that, in such applications, a high rate of convergence would be preferred.

The interpretation of the assumption that $T$ and $T^*$ map the space of polynomials $\mathcal{P}_s$ into itself is made rigorous in Section \ref{sec:representation}, where we restate Theorem \ref{main theorem} as Theorem \ref{thm:formula}. These are ``special cancellation'' or ``vanishing paraproduct'' assumptions that one might like to remove in future work.

Since the circulation of this work, Di Plinio et al.~\cite{DPWW} have presented an alternative ``representation theorem using smooth wavelets'', where they also deal with the paraproduct terms arising from more general cancellation assumptions; see also \cite{DPGW} for an extension of their representation to bi-linear operators. Their version is a closer relative of the continuous wavelet transform, in contrast to the semi-discrete representation in our Theorem \ref{main theorem}.

The paper is organized as follows: in Section \ref{basics} we recall the necessary definitions and results that we are using. Section \ref{sec:representation} is dedicated to a detailed statement of our main result (see Theorem \ref{thm:formula}). As in the case of the dyadic representations using Haar functions, our proof of Theorem \ref{main theorem}/\ref{thm:formula} relies on an expansion (see Proposition \ref{prop:expansion of T}) of the Calder\'on--Zygmund operator in terms of the (previously Haar, now smooth) wavelet basis, but the subsequent analysis of the expansion presents some significant departures from the Haar case. We split the series that appears in this expansion  into five parts which are treated in Section \ref{estimating the different sums}.

\subsection*{Notation} Throughout the paper, we denote by $c, C$ constants that depend at most on some fixed parameters that should be clear from the context. The notation $A\lesssim B$ means that $A\leq CB$ holds for such a constant $C$. Moreover, when $Q$ is a cube and $t>0$, then $tQ$ represents the cube with the same centre and $t$ times the sidelength of $Q$. Also, we make the convection that $\abs{\ }$ stands for the $\ell^{\infty}$ norm on $\R^d$, i.e., $\abs{x}:=\max_{1\leq i\leq d}\abs{x_i}$. While the choice of the norm is not particularly important, this choice is slightly more convenient than the usual Euclidean norm when dealing with cubes as we will: e.g., the diameter of a cube in the $\ell^{\infty}$ norm is equal to its sidelength $\ell(Q)$.

\subsection*{Acknowledgements} Both authors were supported by the Academy of Finland through project No. 314829 and through the Finnish Centre of Excellence in Randomness and Structure ``FiRST'' (project No. 346314). The second author is very grateful to his doctoral supervisor Prof. Tuomas Hyt\"onen for many discussions, giving him a lot of motivation on the subject and plenty of remarks helpful to improve the content and the exposition of this paper. Also, the second author would like to thank the Foundation for Education and  European Culture (Founders Nicos and Lydia Tricha) for their financial support during the academic years 2017--2018, 2018--2019 and 2019--2020.

The authors are grateful to the anonymous referees for their constructive comments that improved our presentation.

\section{Preliminaries}\label{basics}

We recall the following from \cite[Section~2]{Hyt:dyadic op.}. 
\newline The standard (or reference) system of dyadic cubes is
\begin{equation*}
  \mathscr{D}^0:=\{2^{-k}([0,1)^d+l):k\in\Z,l\in\Z^d\}.
\end{equation*}
We will need several dyadic systems, obtained by translating the reference system as follows. Let $\omega=(\omega_j)_{j\in\Z}\in(\{0,1\}^d)^{\Z}$ and
\begin{equation*}
  I\dot+\omega:=I+\sum_{j:2^{-j}<\ell(I)}2^{-j}\omega_j.
\end{equation*}
Then
\begin{equation*}
  \mathscr{D}^{\omega}:=\{I\dot+\omega:I\in\mathscr{D}^0\},
\end{equation*}
and it is straightforward to check that $\mathscr{D}^{\omega}$ inherits the important nestedness property of $\mathscr{D}^0$: if $I,J\in\mathscr{D}^{\omega}$, then $I\cap J\in\{I,J,\varnothing\}$. When the particular $\omega$ is unimportant, the notation $\mathscr{D}$ is sometimes used for a generic dyadic system.

\subsection{Random dyadic systems; good and bad cubes}\label{rand. dyad. systems;good and bad cubes}

We obtain a notion of {\em random dyadic systems} by equipping the parameter set $\Omega:=(\{0,1\}^d)^{\Z}$ with the natural probability measure: each component $\omega_j$ has an equal probability $2^{-d}$ of taking any of the $2^d$ values in $\{0,1\}^d$, and all components are independent of each other. We denote by $\Exp_{\omega}$ the expectation over the random variables $\omega_j,j\in\Z$.

Consider the modulus of continuity $\Theta(t)=t^\theta, \theta\in(0,1)$, for which we will formulate the notion of good and bad cubes.
We also fix a (large) parameter $r\in\Z_+$.
\begin{definition}
A cube $I\in\mathscr{D}^{\omega}$ is called bad if there exists $J\in\mathscr{D}^{\omega}$ such that $\ell(J)\geq 2^r\ell(I)$ and
\begin{equation}\label{eq:bad and good cubes}
  \dist(I,\partial J)\leq\Big(\frac{\ell(I)}{\ell(J)}\Big)^\theta\ell(J):
\end{equation}
roughly, $I$ is relatively close to the boundary of a much bigger cube. A cube is called good if it is not bad.
\end{definition}

We repeat from \cite[Section~2.3]{Hyt:dyadic op.} some basic probabilistic observations related to badness. Let $I\in\mathscr{D}^0$ be a reference interval. The {\em position} of the translated interval
\begin{equation*}
  I\dot+\omega=I+\sum_{j:2^{-j}<\ell(I)}2^{-j}\omega_j,
\end{equation*}
by definition, depends only on $\omega_j$ for $2^{-j}<\ell(I)$. On the other hand, the {\em badness} of $I\dot+\omega$ depends on its {\em relative position} with respect to the bigger intervals
\begin{equation*}
  J\dot+\omega=J+\sum_{j:2^{-j}<\ell(I)}2^{-j}\omega_j+\sum_{j:\ell(I)\leq 2^{-j}<\ell(J)}2^{-j}\omega_j.
\end{equation*}
The same translation component $\sum_{j:2^{-j}<\ell(I)}2^{-j}\omega_j$ appears in both $I\dot+\omega$ and $J\dot+\omega$, and so does not affect the relative position of these intervals. Thus this relative position, and hence the badness of $I$, depends only on $\omega_j$ for $2^{-j}\geq\ell(I)$. In particular:

\begin{lemma}\label{lem:indep}
For $I\in\mathscr{D}^0$, the position and badness of $I\dot+\omega$ are independent random variables.
\end{lemma}

Another observation is the following: by symmetry and the fact that the condition of badness only involves relative position and size of different cubes, it readily follows that the probability of a particular cube $I\dot+\omega$ being bad is equal for all cubes $I\in\mathscr{D}^0$:
\begin{equation*}
  \prob_{\omega}(I\dot+\omega\bad)=\pi_{\bad}=\pi_{\bad}(r,d,\theta).
\end{equation*}

The final observation concerns the value of this probability:

\begin{lemma}\label{lem:value of pi_bad}
We have
\begin{equation*}
  \pi_{\bad}\leq 8d\int_0^{2^{-r}}t^\theta\frac{\ud t}{t}=\frac{8d}{\theta}2^{-r\theta};
\end{equation*}
in particular, $\pi_{\bad}<1$ if $r=r(d,\theta)$ is chosen large enough.
\end{lemma}

The proof of the previous lemma can be found in \cite[Lemma~2.3]{Hyt:dyadic op.}.

\subsection{Wavelet functions}

We introduce the notion of the smooth wavelet functions with compact support associated to any given dyadic system $\mathscr{D}$. Such wavelets were originally constructed by I. Daubechies \cite{ID} but in this paper we will follow \cite{YM}.

In \cite[Chapter~3]{YM} one can find the construction of the smooth wavelets with compact support for $d=1$. Moreover, once the $1$-dimensional wavelets and the related father wavelets $\psi^0=\phi$ are available, the $d$-dimensional wavelets can be constructed by $\psi^{\eta}(x)=\prod_{i=1}^d\psi^{\eta_i}(x_i)$, where $\eta\in\{0,1\}^d\setminus\{0\}$ and we make the convention $\psi^1=\psi$.

\begin{definition}
We say that $\big\{\psi_I^\eta\big\}_{I\in\mathscr{D},\eta\in\{0,1\}^d\setminus\{0\}}$ is a system of wavelets with parameters $(m,u,v)$ if
\begin{equation*}
  \psi_I^{\eta}(x):=2^{{dk}/2}\psi^{\eta}(2^{k}x-l),
\end{equation*}
for some $d$-dimensional wavelet $\psi^{\eta},$ $I=2^{-k}([0,1)^d+l)$, and this collection has the following fundamental properties of a wavelet basis: 
\begin{enumerate}[label=(\roman*)]
  \item being an orthonormal basis of $L^2(\R^d)$
  \item localization: $\supp\psi_I^\eta\subset mI$
  \item regularity: $\abs{\partial^{\alpha}\psi_I^\eta}\leq C{\ell(I)}^{-\abs{\alpha}}\abs{I}^{-1/2}$, for every multi-index $\alpha\in\N$ of order $\abs{\alpha}\leq u$
  \item cancellation: $\int x^\alpha\psi_I^\eta(x)\ud x=0$, when $\abs{\alpha}\leq v.$
\end{enumerate}
\end{definition}
Here $u,v\in\N$ are two parameters that may or may not be equal. Note that Haar functions correspond to $m=1$, $u=v=0$, but in general $m>1$.

For a fixed $\mathscr{D}$, all the wavelet functions $\psi_I^{\eta}$, $I\in\mathscr{D}$ and $\eta\in\{0,1\}^d\setminus\{0\}$, form an orthonormal basis of $L^2(\R^d)$. Hence any function $f\in L^2(\R^d)$ has the orthogonal expansion
\begin{equation*}
  f=\sum_{I\in\mathscr{D}}\sum_{\eta\in\{0,1\}^d\setminus\{0\}}\pair{f}{\psi_I^{\eta}}\psi_I^{\eta}.
\end{equation*}
Since the different $\eta$'s seldom play any major role, this will be often abbreviated (with slight abuse of language) simply as
\begin{equation*}
  f=\sum_{I\in\mathscr{D}}\pair{f}{\psi_I}\psi_I,
\end{equation*}
and the finite summation over $\eta$ is understood implicitly.

\subsection{Wavelet shifts}

A wavelet shift with parameters $i,j\in\N:=\{0,1,2,\ldots\}$ is an operator of the form
\begin{equation*}
  Sf=\sum_{K\in\mathscr{D}}A_K f,\qquad
  A_K f=\sum_{\substack{I,J\in\mathscr{D}:I,J\subseteq K \\ \ell(I)=2^{-i}\ell(K)\\ \ell(J)=2^{-j}\ell(K)}}a_{IJK}\pair{f}{\psi_I}\psi_J,
\end{equation*}
where $\psi_I$ is a wavelet function on $I$ (similarly $\psi_J$), and the $a_{IJK}$ are coefficients with
\begin{equation}\label{eq: bound for the coefficients}
  \abs{a_{IJK}}\leq\frac{\sqrt{\abs{I}\abs{J}}}{\abs{K}}.
\end{equation}

\begin{remark}
The dyadic shifts considered in many other papers correspond to the special case of Haar wavelets. 
\end{remark}

The wavelet shift is called {\em good} if all dyadic cubes $I,J,K$ such that $a_{IJK}\neq0$ satisfy $mI, mJ\subset K$; otherwise, it is called bad. We note that this condition is automatic when $m=1$, but not in general. Nevertheless, a closely related notion of good shifts already appeared in \cite{Hytonen:A2}, where it played a certain role. This notion was not needed in the many works that appeared on this topic since \cite{Hytonen:A2}. The $L^2$ boundedness of the good wavelet shift $S$ is  a consequence of the following facts:

\begin{lemma}\label{lem: pointwise bound for A_K}
If $S$ is a good wavelet shift then $A_K$ indicates an ''averaging operator'' on $K$ which satisfies:
\begin{equation*}
  \abs{A_K f}\lesssim 1_K\frac{1}{|K|}\int_K\abs{f}.
\end{equation*}
\end{lemma}

\begin{proof}
Since $S$ is a good wavelet shift, if $a_{IJK}\neq0$ then $mJ\subset K$ and $mI\subset K$, for fixed $m\ge1$, i.e., $mJ$ and $mI$ are good cubes inside $K$. 

Using the bound (\ref{eq: bound for the coefficients}) for the coefficients $a_{IJK}$, the regularity of $\psi_I$, and the previous fact that $mJ\subset K$, $mI\subset K$, for fixed $m\ge1$, we have
\begin{equation*}
\begin{split}
  |A_K f|&\lesssim\sum_{\substack{I,J\in\mathscr{D}:I,J\subseteq K\\\ell(I)=2^{-i}\ell(K)\\ \ell(J)=2^{-j}\ell(K)}}\frac{\sqrt{\abs{I}\abs{J}}}{\abs{K}}\frac{1_{mJ}}{\sqrt{\abs{J}}}\cdot\int\frac{|f|1_{mI}}{\sqrt{\abs{I}}}\\
  &=\frac{1}{\abs{K}}\Big(\sum_{\substack{J\in\mathscr{D}:J\subseteq K\\ \ell(J)=2^{-j}\ell(K)}}1_{mJ}\Big)\int|f|\Big(\sum_{\substack{I\in\mathscr{D}:I\subseteq K\\\ell(I)=2^{-i}\ell(K)}}1_{mI}\Big)\\
  &\lesssim 1_K\frac{1}{|K|}\int_K\abs{f},
\end{split}
\end{equation*}
where the (easy to check) bounded overlap of the cubes $mJ$ (respectively $mI$) was used in the last step.
\end{proof}

\begin{corollary}\label{corro: L^p bound fo rA_K}
Let $S$ be a good wavelet shift. The following estimate for the ''averaging operator'' $A_K$ holds:
\begin{equation*}
  \Norm{A_K f}{L^p}\lesssim\Norm{f}{L^p},\quad\forall p\in[1,\infty].
\end{equation*}
\end{corollary}

\begin{proof}
Applying the pointwise bound of Lemma \ref{lem: pointwise bound for A_K} to each $A_K$ we have
\begin{equation*}
  \Norm{A_K f}{L^p}\lesssim\BNorm{1_K\frac{1}{|K|}\int_K\abs{f}}{L^p}\lesssim\abs{K}^{1/p}\frac{1}{\abs{K}}\abs{K}^{1/p'}\Norm{f}{L^p}=\Norm{f}{L^p}.
\end{equation*}
\end{proof}

\begin{lemma}
Let $S$ be a good wavelet shift. Then
\begin{equation*}
  \Norm{Sf}{L^2}\lesssim\Norm{f}{L^2}.
\end{equation*}
\end{lemma}

\begin{proof}
We use the orthonormality of the wavelet functions. Let
\begin{equation*}
  {\mathcal{H}}_{K}^{i}:=\lspan\{\psi_I:I\subseteq K, \ell(I)=2^{-i}\ell(K)\},
\end{equation*}
and let $\prob_K^i$ be the orthogonal projection of $L^2$ onto this subspace. For a fixed $i$, these spaces are orthogonal, as $K$ ranges over $\mathscr{D}$.

We have $\pair{f}{\psi_I}=\pair{\prob_K^if}{\psi_I}$ for all $I$ appearing in $A_K$, and hence $A_{K}f=A_K\prob_K^if$. Also, $\psi_J=\prob_K^j\psi_J$ for all $J$ appearing in $A_K$, and hence $A_{K}f=\prob_K^{j}A_K{f}$. We can apply these identities and Pythagoras' theorem to the result that:
\begin{equation*}
\begin{split}
  \Norm{Sf}{L^2}&=\BNorm{\sum_{K\in\mathscr{D}}\prob_K^{j}A_K\prob_K^if}{L^2}\\
  &=\Big(\sum_{K\in\mathscr{D}}\Norm{\prob_K^{j}A_K\prob_K^if}{L^2}^2\Big)^{1/2}\\
  &\lesssim\Big(\sum_{K\in\mathscr{D}}\Norm{\prob_K^{i}f}{L^2}^2\Big)^{1/2}\\
  &\lesssim\Norm{f}{L^2},
\end{split}
\end{equation*}
where we used the $L^2$ boundedness of $A_K$ from Corollary \ref{corro: L^p bound fo rA_K} in the second-to-last step.
\end{proof}

\section{The dyadic representation theorem for smooth compactly supported wavelets}\label{sec:representation}

Let $T$ be a Calder\'on--Zygmund operator on $\R^d$. That is, it acts on a suitable dense subspace of functions in $L^2(\R^d)$ (for the present purposes, this class should at least contain the indicators of cubes in $\R^d$) and has the kernel representation
\begin{equation*}
  Tf(x)=\int_{\R^d}K(x,y)f(y)\ud y,\qquad x\notin\supp f.
\end{equation*}
Moreover, we assume that the kernel is $s$-times differentiable and satisfies the {\em higher order standard estimate}:
\begin{equation}\label{eq:high. ord. st. estim.}
\begin{split}
  \abs{\partial^{\alpha}K(x,y)}&\leq\frac{C_1}{\abs{x-y}^{d+\abs{\alpha}}}
\end{split}
\end{equation}
for all $x,y\in\R^d$, $x\neq y$, $\alpha\in\N$ and $\abs{\alpha}\leq s$. Let us denote the smallest admissible constant $C_1$ by $\Norm{K}{CZ_s}$.

We say that $T$ is a bounded Calder\'on--Zygmund operator, if in addition $T:L^2(\R^d)\to L^2(\R^d)$, and we denote its operator norm by $\Norm{T}{L^2\to L^2}$.

%%%

Under such assumptions, we can also define the action of $T$ on the space $\mathcal{P}_s$ of polynomials of degree less than $s$. This is well known, and the reader can consult \cite{Torres} (see also \cite{PT}) for a comprehensive discussion. The necessary set-up for our needs is as follows:

If $\psi\in C_c^s(B(0,R))$ satisfies $\int_{\R^d}P\psi=0$ for all $P\in\mathcal{P}_s$, then we have, for $\abs{x}>2R$,
\begin{equation*}
  T\psi(x)=\int_{\R^d} K(x,y)\psi(y)\ud y 
  =\int_{B(0,R)} \Big(K(x,y)-\sum_{\mathclap{\substack{0\leq|\alpha|<s}}}\frac{y^{\alpha}}{\alpha!}\partial_2^{\alpha}K(x,0)\Big)\psi(y)\ud y
\end{equation*}
and hence
\begin{multline*}
  \abs{T\psi(x)}\leq\int_{B(0,R)} s\Big(\int_{0}^{1}\sum_{\substack{|\alpha|=s}}\frac{\abs{y}^{\abs{\alpha}}}{\alpha!}\abs{\partial_2^{\alpha}K(x,ty)}(1-t)^{s-1}\ud t\Big)\abs{\psi(y)}\ud y \\
  \lesssim\Norm{K}{CZ_s}\int_{B(0,R)} R^s\Big(\int_0^1\frac{s}{\abs{x-ty}^{d+s}}(1-t)^{s-1}\ud t\Big)\abs{\psi(y)}\ud y \\
  \lesssim\Norm{K}{CZ_s}\frac{R^s}{\abs{x}^{d+s}}\Norm{\psi}{1}.
\end{multline*}
This expression is integrable against any $P\in\mathcal{P}_s$ over the region $B(0,2R)^c$. On the other hand, it is clear that $T\psi\in L^2(\R^d)\subset L^1_{\operatorname{loc}}(\R^d)$ is also integrable against $P\in\mathcal{P}_s \subset L^\infty_{\operatorname{loc}}(\R^d)$ over $B(0,2R)$. Thus
\begin{equation*}
  \pair{T^* P}{\psi}:=\pair{P}{T\psi}:=\int_{\R^d}P(x)T\psi(x)\ud x
\end{equation*}
is well defined for every $P\in\mathcal{P}_s$ and every $\psi\in C_c^s(\R^d)$ that is orthogonal to $\mathcal{P}_s$. This defines $T^*P$ as a functional on the said subspace of $C_c^s(\R^d)$, and the definition of $TP$ can be given in a similar way, since the adjoint $T^*$ satisfies the same assumptions.

We say that $T$ maps $\mathcal{P}_s$ into itself, if $\pair{TP}{\psi}=0$ for all $P\in\mathcal{P}_s$ and all $\psi\in C_c^s(\R^d)$ that are orthogonal to $\mathcal{P}_s$.

%%%

Here is our main result:
\begin{theorem}\label{thm:formula}
Let $T$ be a bounded Calder\'on--Zygmund operator with a kernel satisfying (\ref{eq:high. ord. st. estim.})  and suppose that both $T$ and $T^*$ map $\mathcal{P}_s$ into itself, in the sense defined above.
Moreover, let the wavelet $\psi_I$ satisfy the regularity and cancellation property for $u=s$ and $v=s-1$, respectively. Then for any given $\epsilon>0$, $T$ has an expansion, say for $f,g\in C^1_c(\R^d)$,
\begin{equation*}
  \pair{g}{Tf}=c\cdot\big(\Norm{K}{CZ_s}+\Norm{T}{L^2\to L^2}\big)\cdot\Exp_{\omega} \sum_{i,j=1}^{\infty}2^{-(s-\epsilon)\max(i,j)}\pair{g}{S^{ij}_{\omega}f},
\end{equation*}
where $c$ depends only on d, s and $\epsilon$, and $S^{ij}_{\omega}$ is a good wavelet shift of parameters $(i,j)$ on the dyadic system $\mathscr{D}^{\omega}$.
\end{theorem}

The following remark shows that the assumption that $T$ and $T^*$ map $\mathcal{P}_s$ into itself follows from the other assumptions of Theorem \ref{thm:formula}, if in addition the operator $T$ is translation invariant.

\begin{remark}
Let $T$ be a bounded Calder\'on--Zygmund operator with a kernel satisfying (\ref{eq:high. ord. st. estim.}) and suppose in addition that $T$ is translation invariant.
Then both $T$ and $T^*$ map $\mathcal{P}_s$ into itself.
\end{remark}

\begin{proof}
It is enough to consider just $T$, since all the assumptions, and hence the conclusions, pass to the adjoint $T^*$. For the result concerning $T$, we refer the reader to \cite[Proposition 2.2.17]{Torres}.
\end{proof}

A key to the proof of the dyadic representation is a random expansion of $T$ in terms of wavelet functions $\psi_I$, where the bad cubes are avoided:

\begin{proposition}\label{prop:expansion of T}
Let $T\in\bddlin(L^2(\R^d))$ and $f\in C^1_c(\R^d)$, $g\in C^1_c(\R^d)$. Then the following representation is valid:
\begin{equation*}
  \pair{g}{Tf}=\frac{1}{\pi_{\good}}\Exp_{\omega}\sum_{I,J\in\mathscr{D}^{\omega}}1_{\good}(\operatorname{smaller}\{I,J\})\cdot\pair{g}{\psi_{J}}\pair{\psi_{J}}{T\psi_{I}}\pair{\psi_{I}}{f},
\end{equation*}
where
\begin{equation*}
  \operatorname{smaller}\{I,J\}:=\begin{cases} I & \text{if }\ell(I)\leq\ell(J) \\ J & \text{if }\ell(I)>\ell(J), \end{cases}
\end{equation*}
and $\pi_{\good}:=1-\pi_{\bad}>0$.
\end{proposition}

\begin{proof}
The proof is analogous to the one given in \cite[Proposition~3.5]{Hyt:dyadic op.}, replacing the Haar functions $h_I$ and $h_J$ should be replaced by the wavelet functions $\psi_I$ and $\psi_J$, respectively. We provide the details for the convenience of the reader.

Recall that
\begin{equation*}
  f=\sum_{I\in\mathscr{D}^0}\pair{f}{\psi_{I\dot+\omega}}\psi_{I\dot+\omega}
\end{equation*}
for any fixed $\omega\in\Omega$; and we can also take the expectation $\Exp_{\omega}$ of both sides of this identity.

Let
\begin{equation*}
  1_{\good}(I\dot+\omega):=\begin{cases} 1, & \text{if $I\dot+\omega$ is good},\\ 0, & \text{else}\end{cases}
\end{equation*}
We make use of the above random wavelet expansion of $f$, multiply and divide by
\begin{equation*}
  \pi_{\good}=\prob_{\omega}(I\dot+\omega\good)=\Exp_{\omega}1_{\good}(I\dot+\omega),
\end{equation*}
and use the independence from Lemma~\ref{lem:indep} to get:
\begin{equation*}
\begin{split}
  \pair{g}{Tf}
  &=\Exp_{\omega}\sum_{I}\pair{g}{T\psi_{I\dot+\omega}}\pair{\psi_{I\dot+\omega}}{f} \\
  &=\frac{1}{\pi_{\good}}\sum_{I}\Exp_{\omega}[1_{\good}(I\dot+\omega)] \Exp_{\omega}[\pair{g}{T\psi_{I\dot+\omega}}\pair{\psi_{I\dot+\omega}}{f}] \\
  &=\frac{1}{\pi_{\good}}\Exp_{\omega}\sum_{I}1_{\good}(I\dot+\omega) \pair{g}{T\psi_{I\dot+\omega}}\pair{\psi_{I\dot+\omega}}{f}  \\
  &=\frac{1}{\pi_{\good}}\Exp_{\omega}\sum_{I,J}1_{\good}(I\dot+\omega) \pair{g}{\psi_{J\dot+\omega}}\pair{\psi_{J\dot+\omega}}{T\psi_{I\dot+\omega}}\pair{\psi_{I\dot+\omega}}{f}.
\end{split}
\end{equation*}
On the other hand, using independence again in half of this double sum, we have
\begin{equation*}
\begin{split}
  &\frac{1}{\pi_{\good}}\sum_{\ell(I)>\ell(J)}\Exp_{\omega}[1_{\good}(I\dot+\omega) \pair{g}{\psi_{J\dot+\omega}}\pair{\psi_{J\dot+\omega}}{T\psi_{I\dot+\omega}}\pair{\psi_{I\dot+\omega}}{f} ] \\
  &=\frac{1}{\pi_{\good}}\sum_{\ell(I)>\ell(J)}\Exp_{\omega}[1_{\good}(I\dot+\omega)]
     \Exp_{\omega}[ \pair{g}{\psi_{J\dot+\omega}}\pair{\psi_{J\dot+\omega}}{T\psi_{I\dot+\omega}}\pair{\psi_{I\dot+\omega}}{f} ] \\
   &= \Exp_{\omega}\sum_{\ell(I)>\ell(J)}
    \pair{g}{\psi_{J\dot+\omega}}\pair{\psi_{J\dot+\omega}}{T\psi_{I\dot+\omega}}\pair{\psi_{I\dot+\omega}}{f},
\end{split}
\end{equation*}
and hence
\begin{equation*}
\begin{split}
    \pair{g}{Tf}
    &= \frac{1}{\pi_{\good}}\Exp_{\omega}\sum_{\ell(I)\leq\ell(J)}
        1_{\good}(I\dot+\omega) \pair{g}{\psi_{J\dot+\omega}}\pair{\psi_{J\dot+\omega}}{T\psi_{I\dot+\omega}}\pair{\psi_{I\dot+\omega}}{f} \\
     &\qquad+\Exp_{\omega}\sum_{\ell(I)>\ell(J)}
        \pair{g}{\psi_{J\dot+\omega}}\pair{\psi_{J\dot+\omega}}{T\psi_{I\dot+\omega}}\pair{\psi_{I\dot+\omega}}{f}.
\end{split}
\end{equation*}
Comparison with the basic identity
\begin{equation}\label{eq:basic}
  \pair{g}{Tf}
  =\Exp_{\omega}\sum_{I,J}\pair{g}{\psi_{J\dot+\omega}}\pair{\psi_{J\dot+\omega}}{T\psi_{I\dot+\omega}}\pair{\psi_{I\dot+\omega}}{f}
\end{equation}
shows that
\begin{equation*}
\begin{split}
  &\Exp_{\omega}\sum_{\ell(I)\leq\ell(J)}
        \pair{g}{\psi_{J\dot+\omega}}\pair{\psi_{J\dot+\omega}}{T\psi_{I\dot+\omega}}\pair{\psi_{I\dot+\omega}}{f} \\
    &= \frac{1}{\pi_{\good}}\Exp_{\omega}\sum_{\ell(I)\leq\ell(J)}
        1_{\good}(I\dot+\omega) \pair{g}{\psi_{J\dot+\omega}}\pair{\psi_{J\dot+\omega}}{T\psi_{I\dot+\omega}}\pair{\psi_{I\dot+\omega}}{f}.
\end{split}  
\end{equation*}
Symmetrically, we also have
\begin{equation*}
\begin{split}
  &\Exp_{\omega}\sum_{\ell(I)>\ell(J)}
        \pair{g}{\psi_{J\dot+\omega}}\pair{\psi_{J\dot+\omega}}{T\psi_{I\dot+\omega}}\pair{\psi_{I\dot+\omega}}{f} \\
    &= \frac{1}{\pi_{\good}}\Exp_{\omega}\sum_{\ell(I)>\ell(J)}
        1_{\good}(J\dot+\omega) \pair{g}{\psi_{J\dot+\omega}}\pair{\psi_{J\dot+\omega}}{T\psi_{I\dot+\omega}}\pair{\psi_{I\dot+\omega}}{f},
\end{split}  
\end{equation*}
and this completes the proof.
\end{proof}

For the analysis of the series appearing in Proposition \ref{prop:expansion of T} we recall the notion of the long distance \cite[Definition~6.3]{NTV}
\begin{equation*}
  D(I,J):=\ell(I)+\dist(I,J)+\ell(J). 
\end{equation*}
We focus on the summation inside $\Exp_{\omega}$, for a fixed value of $\omega\in\Omega$, and manipulate it into the required form. Moreover, we will focus on the half of the sum with $\ell(J)\geq\ell(I)$, the other half being handled symmetrically. We further divide this sum into the following parts:
\begin{equation*}
\begin{split}
  \sum_{\ell(I)\leq\ell(J)}&=\sum_{\substack{\dist(I,J)>\ell(J)(\ell(I)/\ell(J))^\theta\\ \dist(mI,mJ)>\frac{1}{2}D(I,J)}}+\sum_{\substack{\dist(I,J)>\ell(J)(\ell(I)/\ell(J))^\theta\\ \dist(mI,mJ)\leq\frac{1}{2}D(I,J)}}+\sum_{I\subsetneq J}+\sum_{I=J}\\
  &\qquad+\sum_{\substack{\dist(I,J)\leq\ell(J)(\ell(I)/\ell(J))^\theta\\ I\cap J=\varnothing}}\\
  &=:\sigma_{\operatorname{far}}+\sigma_{\operatorname{between}}+\sigma_{\operatorname{in}}+\sigma_{=}+\sigma_{\operatorname{near}}.
\end{split}
\end{equation*}
We observe that the main difference in the division of the previous sum and the one in \cite[after the Proposition~3.5]{Hyt:dyadic op.} is that the sum $\sigma_{\operatorname{out}}$ in \cite{Hyt:dyadic op.} has been split into $\sigma_{\operatorname{far}}$ and $\sigma_{\operatorname{between}}$, which are handled differently. Regarding the sum $\sigma_{\operatorname{in}}$ we will not use the same method as in \cite[Section~3.2]{Hyt:dyadic op.}. The sums  $\sigma_{\operatorname{=}}$ and $\sigma_{\operatorname{near}}$ will be treated in a similar but not exactly the same way as in \cite[Section~3.3]{Hyt:dyadic op.}.

In order to recognize these series as sums of good wavelet shifts, we need to find, for each pair $(I,J)$ appearing here, a common dyadic ancestor which contains $mI$ and $mJ$. The following lemma provides the existence of such containing cubes, with control on their size:

\begin{lemma}\label{lem:IveeJ}
If $I\in\mathscr{D}$ is good and $J\in\mathscr{D}$ is a cube with $\ell(J)\geq\ell(I)$, then there exists $K\supseteq mI\cup mJ$ which satisfies
\begin{equation*}
\begin{split}
  \ell(K)\Big(\frac{\ell(I)}{\ell(K)}\Big)^\theta&\lesssim D(I,J),\qquad\text{always,}\qquad\qquad\text{and}\\
  \ell(K)\lesssim\ell(I),\qquad\text{if}\qquad\dist(I,J)&\leq\ell(J)\Big(\frac{\ell(I)}{\ell(J)}\Big)^\theta\qquad\text{and}\qquad J\cap I=\varnothing.
\end{split}
\end{equation*}
\end{lemma}

\begin{proof}
Let us start with the following initial observation: if $I\in\mathscr{D}$ is good and $K\in\mathscr{D}$ satisfies $I\subseteq K$, and $\ell(K)\ge2^r\ell(I)$, then
\begin{equation*}
  \dist(I,K^c)=\dist(I,\partial K)>\ell(K)\Big(\frac{\ell(I)}{\ell(K)}\Big)^\theta=\ell(K)^{1-\theta}\ell(I)^{\theta}\ge2^{r(1-\theta)}\ell(I)>m\ell(I),
\end{equation*}
when $r$ is large enough. Hence $mI\subseteq K$, and we can proceed with the proof of $mJ\subseteq K$.
Using an elementary triangle inequality we estimate $\dist(I,K^c)$ in the following way:
\begin{equation*}
\begin{split}
  \dist(I,K^c)&\leq\dist(I,mJ)+\ell(mJ)+\dist(mJ,K^c)\\
  &\leq\dist(I,J)+m\ell(J)+\dist(mJ,K^c).
\end{split}
\end{equation*}
Thus, 
\begin{equation}\label{eq 1: mJ is good cube in K}
\begin{split}
  \dist(mJ,K^c)&\ge\dist(I,K^c)-\dist(I,J)-m\ell(J)\\
  &>\ell(K)\Big(\frac{\ell(I)}{\ell(K)}\Big)^\theta-\dist(I,J)-m\ell(J).
\end{split}
\end{equation}
In order to conclude that $mJ\subseteq K$ we want the right hand side of (\ref{eq 1: mJ is good cube in K}) to be non-negative. This is achieved by taking the smallest $\ell(K)$ such that
\begin{equation*}
  \ell(K)\Big(\frac{\ell(I)}{\ell(K)}\Big)^\theta\ge\dist(I,J)+m\ell(J).
\end{equation*}
Then, in fact 
\begin{equation}\label{eq 2: mJ is good cube in K}
  \ell(K)\Big(\frac{\ell(I)}{\ell(K)}\Big)^\theta\lesssim\dist(I,J)+m\ell(J)\lesssim\dist(I,J)+\ell(J).
\end{equation}
Hence,
\begin{equation*}
  \ell(K)\Big(\frac{\ell(I)}{\ell(K)}\Big)^\theta\lesssim D(I,J).
\end{equation*}
This proves the first estimate.

\subsubsection*{Case $\dist(I,J)\leq\ell(J)(\ell(I)/\ell(J))^\theta$ and $I\cap J=\varnothing$:}
As $I\cap J=\varnothing$, we have $\dist(I,J)=\dist(I,\partial J)$, and since $I$ is good, this implies $\ell(J)<2^r\ell(I)$. We can then dominate the right hand side of (\ref{eq 2: mJ is good cube in K}) by
\begin{equation}\label{eq 3: mJ is good cube in K}
  \ell(J)(\ell(I)/\ell(J))^\theta+\ell(J)\lesssim\ell(J)\lesssim\ell(I).
\end{equation}
Thus, from (\ref{eq 2: mJ is good cube in K}) and (\ref{eq 3: mJ is good cube in K}) we have
\begin{equation*}
  \frac{\ell(K)}{\ell(I)}\Big(\frac{\ell(I)}{\ell(K)}\Big)^\theta\lesssim1\quad\text{and}\quad\ell(K)\lesssim\ell(I),
\end{equation*}
so this proves the second estimate.
\end{proof}
We denote the minimal such $K$ by $I\vee J$, thus
\begin{equation*}
  I\vee J:=\bigcap_{K\supseteq mI\cup mJ} K.
\end{equation*}

\section{Estimates for the different sums $\sigma_{\operatorname{far}}, \sigma_{\operatorname{between}}, \sigma_{\operatorname{in}}, \sigma_{=}, \sigma_{\operatorname{near}}$}\label{estimating the different sums}

\subsection{Far away cubes, $\sigma_{\operatorname{far}}$}

We reorganize the sum $\sigma_{\operatorname{far}}$ with respect to the new summation variable $K=I\vee J$, as well as the relative size of $I$ and $J$ with respect to $K$:
\begin{equation*}
  \sigma_{\operatorname{far}}=\sum_{j=1}^{\infty}\sum_{i=j}^{\infty}\sum_K\sum_{\substack{I,J:\dist(I,J)>\ell(J)(\ell(I)/\ell(J))^\theta\\ \dist(mI,mJ)>\frac{1}{2}D(I,J)\\ I\vee J=K\\\ell(I)=2^{-i}\ell(K),\ell(J)=2^{-j}\ell(K)}}.
\end{equation*}
Note that we can start the summation from $1$ instead of $0$, since the disjointness of $I$ and $J$ implies that $K=I\vee J$ must be strictly larger than either of $I$ and $J$. The goal is to identify the quantity in parentheses as a decaying factor times an averaging operator with parameters $(i,j)$. The proof of the following lemma is similar to \cite[Lemma~3.8]{Hyt:dyadic op.} but to make use of the smoothness, we subtract a higher order Taylor expansion of the kernel $K$ instead of $y\mapsto K(x,y)$ at $y=c_I$.

\begin{lemma}\label{lem:far away cubes}
For $I$ and $J$ appearing in $\sigma_{\operatorname{far}}$, we have
\begin{equation*}
  \abs{\pair{\psi_J}{T\psi_I}}\lesssim\Norm{K}{CZ_s}\frac{\sqrt{\abs{I}\abs{J}}}{\abs{K}}\Big(\frac{\ell(I)}{\ell(K)}\Big)^{-\theta(d+s)}\Big(\frac{\ell(I)}{\ell(K)}\Big)^{s},
\end{equation*}
where $K=I\vee J$ and $\theta\in(0,1)$.
\end{lemma}

\begin{proof}
Using the properties of $\psi_I$, Taylor series of order $s$ of $y\mapsto K(x,y)$ at the centre point $y=c_{I}$ of $I$, higher order standard estimate of the kernel (\ref{eq:high. ord. st. estim.}), and Lemma \ref{lem:IveeJ}
\begin{multline*}
  \abs{\pair{\psi_J}{T\psi_I}}=\Babs{\iint \psi_J(x)K(x,y)\psi_I(y)\ud y\ud x}\\
  =\Babs{\iint \psi_J(x)\Big(K(x,y)-\sum_{\mathclap{\substack{0\leq|\alpha|<s}}}\frac{(y-c_I)^{\alpha}}{\alpha!}\partial_2^{\alpha}K(x,c_I)\Big)\psi_I(y)\ud y\ud x}\\
  \leq\iint s\abs{\psi_{J}(x)}\Big(\int_{0}^{1}\sum_{\substack{|\alpha|=s}}\frac{\abs{y-c_I}^{\abs{\alpha}}}{\alpha!}\abs{\partial_2^{\alpha}K(x,ty+(1-t)c_I)}(1-t)^{s-1}\ud t\Big)\abs{\psi_{I}(y)}\ud y\ud x\\
  \lesssim\Norm{K}{CZ_s}\iint\abs{\psi_{J}(x)}\ell(I)^s\Big(\int_0^1\frac{s}{\abs{x-(c_I+t(y-c_I))}^{d+s}}(1-t)^{s-1}\ud t\Big)\abs{\psi_{I}(y)}\ud y\ud x\\
  \lesssim\Norm{K}{CZ_s}\frac{\ell(I)^s}{\dist(mI,mJ)^{d+s}}\Norm{\psi_J}{1}\Norm{\psi_I}{1}\\
  \lesssim\Norm{K}{CZ_s}\frac{\ell(I)^s}{D(I,J)^{d+s}}\Norm{\psi_J}{1}\Norm{\psi_I}{1}\\
  \lesssim\Norm{K}{CZ_s}\frac{\ell(I)^s}{\ell(K)^{d+s}}\Big(\frac{\ell(I)}{\ell(K)}\Big)^{-\theta(d+s)}\Norm{\psi_J}{1}\Norm{\psi_I}{1}\\
  \lesssim\Norm{K}{CZ_s}\frac{1}{\ell(K)^d}\Big(\frac{\ell(I)}{\ell(K)}\Big)^s\Big(\frac{\ell(I)}{\ell(K)}\Big)^{-\theta(d+s)}|mJ||mI||J|^{-\frac{1}{2}}|I|^{-\frac{1}{2}}\\
  \lesssim\Norm{K}{CZ_s}\frac{1}{\ell(K)^d}\Big(\frac{\ell(I)}{\ell(K)}\Big)^s\Big(\frac{\ell(I)}{\ell(K)}\Big)^{-\theta(d+s)}\sqrt{\abs{J}}\sqrt{\abs{I}}.
\end{multline*}
\end{proof}

\begin{lemma}
\begin{equation*}
\begin{split}
  \sum_{\substack{I,J:\dist(I,J)>\ell(J)(\ell(I)/\ell(J))^\theta\\ \dist(mI,mJ)>\frac{1}{2}D(I,J)\\I\vee J=K\\\ell(I)=2^{-i}\ell(K)\leq \ell(J)=2^{-j}\ell(K)}}
  &1_{\good}(I)\cdot  \pair{g}{\psi_{J}}\pair{\psi_{J}}{T\psi_{I}}\pair{\psi_{I}}{f}\\
  &=c\Norm{K}{CZ_s}2^{i\theta(d+s)}2^{-is}\pair{g}{A_K^{ij}f},
\end{split}
\end{equation*}
where $\theta\in(0,1)$ and $A_K^{ij}$ is an averaging operator with parameters $(i,j)$. 
\end{lemma}

\begin{proof}
By Lemma \ref{lem:far away cubes}, substituting $\ell(I)/\ell(K)=2^{-i}$,
\begin{equation*}
  \abs{\pair{\psi_J}{T\psi_I}}\lesssim\Norm{K}{CZ_s}\frac{\sqrt{\abs{I}\abs{J}}}{\abs{K}}2^{i\theta(d+s)}2^{-is},
\end{equation*}
and the first factor is precisely the required size of the coefficients of $A_K^{ij}$.
\end{proof}

Summarizing, we have
\begin{equation*}
  \sigma_{\operatorname{far}}=c\Norm{K}{CZ_s}\sum_{j=1}^{\infty}\sum_{i=j}^{\infty}2^{-i(s-\epsilon)}\pair{g}{S^{ij}f},
\end{equation*}
where we choose $\theta=\frac{\epsilon}{d+s}$ for any given $\epsilon>0$ and $S^{ij}$ is a good wavelet shift with parameters $(i,j)$. 

\subsection{Intermediate cubes, $\sigma_{\operatorname{between}}$}

Let $M>m$. In this part, we make use of the fact that $\psi_J$ has a Taylor series of order $s$ at the centre point $c_I$ of $I$ and we denote
\begin{equation*}
  \text{Tayl}_{s}(\psi_J,c_I):=\sum_{\substack{0\leq|\alpha|<s}}\frac{(x-c_I)^{\alpha}}{\alpha!}\partial^{\alpha}\psi_J(c_I). 
\end{equation*}
We drop $c_I$, when it is clear from the context. Thus, we have
\begin{equation}\label{eq1:sigma between and sigma in}
  \pair{\psi_J}{T\psi_I}=\pair{\psi_J-\text{Tayl}_{s}(\psi_J,c_I)}{T\psi_I}+\pair{\text{Tayl}_{s}(\psi_J,c_I)}{T\psi_I}.
\end{equation}

Observe that due to the hypothesis of Theorem \ref{thm:formula} that the operators $T,T^*$ map $\mathcal{P}_s$ to itself, and by the cancellation of $\psi_I$ the last term of (\ref{eq1:sigma between and sigma in}) vanishes. The first term of (\ref{eq1:sigma between and sigma in}) we can further split as
\begin{equation}\label{eq2:sigma between and sigma in}
\begin{split}
  \pair{\psi_J}{T\psi_I} &=\pair{1_{(MI)^c}(\psi_J-\text{Tayl}_{s}(\psi_J,c_I))}{T\psi_I} \\
  &\qquad+\pair{1_{MI}(\psi_J-\text{Tayl}_{s}(\psi_J,c_I))}{T\psi_I}.
\end{split}
\end{equation}

In the following we estimate the remaining non-vanishing terms of (\ref{eq2:sigma between and sigma in}). For these terms, we obtain estimates that do not depend on the fact that we are dealing with the intermediate cubes, and in fact we will use these same estimates again to deal with $\sigma_{\operatorname{in}}$.
\begin{lemma}\label{eq:first term}
For all $I,J\in\mathscr{D}$ such that $\ell(I)\leq\ell(J)$, we have
\begin{equation*}
  \abs{\pair{1_{(MI)^c}(\psi_J-\text{Tayl}_{s}(\psi_J,c_I))}{T\psi_I}}\lesssim\Norm{K}{CZ_s}\Big(\frac{\abs{I}}{\abs{J}}\Big)^{1/2}\Psi\Big(\frac{\ell(I)}{\ell(J)}\Big),
\end{equation*}
where
\begin{equation*}
  \Psi(t):=t^s\Big(\log\frac{1}{t}+1\Big)\lesssim t^{s-\epsilon}\quad\text{for any given}\quad\epsilon>0\quad\text{and}\quad t\in(0,1].
\end{equation*}
\end{lemma}

\begin{proof}
Let us denote $\text{Tayl}_{s}(K):=\sum_{\substack{0\leq|\alpha|<s}}\frac{(y-c_I)^{\alpha}}{\alpha!}\partial^{\alpha}_2K(x,c_I)$. Using the cancellation of $\psi_I$, the Taylor series of order $s$ of $y\mapsto K(x,y)$ at the centre point $y=c_{I}$ of $I$ and the higher order standard estimate of the kernel (\ref{eq:high. ord. st. estim.})
\begin{multline}\label{eq:Taylor 1}
  \abs{\pair{1_{(MI)^c}(\psi_J-\text{Tayl}_{s}(\psi_J,c_I))}{T\psi_I}}\\
  \leq\iint1_{(MI)^{c}}(x)\abs{\psi_J(x)-\text{Tayl}_{s}(\psi_J,c_I)}\abs{K(x,y)-\text{Tayl}_{s}(K)}\abs{\psi_{I}(y)}\ud y\ud x\\
  \lesssim\Norm{K}{CZ_s}\ell(I)^s\|\psi_I\|_1\int1_{(MI)^c}(x)\frac{\abs{\psi_J(x)-\text{Tayl}_{s}(\psi_J,c_I)}}{\dist(x,mI)^{d+s}}\ud x.
\end{multline}

Now, using the regularity property and the Taylor series of order $s$ of $\psi_J$ at the centre point $c_I$ of $I$ we derive the following two estimates:
\subsubsection*{Estimate 1}

\begin{equation*}
\begin{split}
  \abs{\psi_J(x)-\text{Tayl}_{s}(\psi_J,c_I)}&\leq\|\psi_J\|_{\infty}+\sum_{\mathclap{\substack{0\leq|\alpha|<s}}}\frac{\abs{x-c_I}^{\abs{\alpha}}}{\alpha!}\|\partial^{\alpha}\psi_J\|_{\infty}\\
  &\lesssim\abs{J}^{-1/2}\Big(\sum_{\substack{a<s}}\frac{\dist(x,mI)^{a}}{\ell(J)^a}\Big),
\end{split}
\end{equation*}
where $a=\abs{\alpha}\in\N$.

\subsubsection*{Estimate 2}

\begin{equation*}
\begin{split}
  \abs{\psi_J(x)-\text{Tayl}_{s}(\psi_J,c_I)}&\leq s\int_{0}^{1}\sum_{\mathclap{\substack{|\alpha|=s}}}\frac{\abs{x-c_I}^{\abs{\alpha}}}{\alpha!}\abs{\partial^{\alpha}\psi_{J}(tx+(1-t)c_{I})}(1-t)^{s-1}\ud t\\
  &\lesssim \abs{x-c_I}^s\|\partial^{s}\psi_J\|_{\infty}\\
  &\lesssim\abs{J}^{-1/2}\frac{{\dist(x,mI)}^s}{\ell(J)^s}.
\end{split}
\end{equation*}
Thus, the right hand side of (\ref{eq:Taylor 1}) is dominated by

\begin{equation*}
\begin{split}
  &\lesssim\Norm{K}{CZ_s}\ell(I)^s\frac{\|\psi_I\|_1}{\abs{J}^{1/2}}\Big(\int_{\substack{(MI)^c\\ \dist(x,mI)\leq\ell(J)}}\frac{\dist(x,mI)^s}{\ell(J)^s}\frac{1}{\dist(x,mI)^{d+s}}\ud x\\
  &\qquad+\int_{\dist(x,mI)>\ell(J)}\frac{\dist(x,mI)^{s-1}}{\ell(J)^{s-1}}\frac{1}{\dist(x,mI)^{d+s}}\ud x\Big)\\
  &\lesssim\Norm{K}{CZ_s}\ell(I)^s\Big(\frac{\abs{I}}{\abs{J}}\Big)^{1/2}\Big(\int_{\ell(I)}^{\ell(J)}\frac{1}{\ell(J)^s}\frac{1}{t}\ud t
  +\int_{\ell(J)}^{\infty}\frac{1}{\ell(J)^{s-1}}\frac{1}{t^2}\ud t\Big)\\
  &=\Norm{K}{CZ_s}\ell(I)^s\Big(\frac{\abs{I}}{\abs{J}}\Big)^{1/2}\Big(\frac{1}{\ell(J)^s}\log\frac{\ell(J)}{\ell(I)}+\frac{1}{\ell(J)^{s-1}}\frac{1}{\ell(J)}\Big)\\
  &=\Norm{K}{CZ_s}\Big(\frac{\abs{I}}{\abs{J}}\Big)^{1/2}\Psi\Big(\frac{\ell(I)}{\ell(J)}\Big).
\end{split}
\end{equation*}
\end{proof}

\begin{lemma}\label{eq:second term}
For all $I,J\in\mathscr{D}$ such that $\ell(I)\leq\ell(J)$, we have
\begin{equation}\label{eq:Taylor 2}
  \abs{\pair{1_{MI}(\psi_J-\text{Tayl}_{s}(\psi_J,c_I))}{T\psi_I}}\lesssim\Norm{T}{L^2\to L^2}\Big(\frac{\ell(I)}{\ell(J)}\Big)^s\Big(\frac{\abs{I}}{\abs{J}}\Big)^{1/2}.
\end{equation}
\end{lemma}

\begin{proof}
By the Taylor series of order $s$ of $\psi_J$ at the centre point $c_I$ of $I$ and the regularity properties of $\psi_I,\psi_J$, we can compute the left hand side of (\ref{eq:Taylor 2}) as follows:
\begin{multline*}
  \Babs{\Bpair{1_{MI}\int_0^1\sum_{\mathclap{\substack{|\alpha|=s}}}\frac{{(x-c_I)}^{\alpha}}{\alpha!}\partial^{\alpha}\psi_J(tx+(1-t)c_I)(1-t)^{s-1}\ud t}{T\psi_I}}\\
  \lesssim\Norm{T}{L^2\to L^2}(M\ell(I))^s\|\partial^{s}\psi_J\|_{\infty}\abs{I}^{1/2}\\
  \lesssim\Norm{T}{L^2\to L^2}\ell(I)^s\ell(J)^{-s}\abs{J}^{-1/2}\abs{I}^{1/2}\\
  =\Norm{T}{L^2\to L^2}\Big(\frac{\ell(I)}{\ell(J)}\Big)^s\Big(\frac{\abs{I}}{\abs{J}}\Big)^{1/2}.
\end{multline*}
\end{proof}
By combining equation (\ref{eq2:sigma between and sigma in}), Lemmata \ref{eq:first term} and \ref{eq:second term} we have 
\begin{equation}\label{eq3:sigma between and sigma in}
  \abs{\pair{\psi_J}{T\psi_I}}\lesssim(\Norm{K}{CZ_s}+\Norm{T}{L^2\to L^2})\Big(\frac{|I|}{|J|}\Big)^{1/2}\Big(\frac{\ell(I)}{\ell(J)}\Big)^{s-\epsilon}.
\end{equation}

Now, for the completion of the analysis of the sum $\sigma_{\operatorname{between}}$ we will need the following lemma:

\begin{lemma}\label{lem:an estimate for D(I,J)}
For $I$ and $J$ appearing in $\sigma_{\operatorname{between}}$, we have
\begin{equation*}
  D(I,J)\lesssim\ell(J),  
\end{equation*}
where $D(I,J)$ is the long distance introduced in Section \ref{sec:representation}.
\begin{proof}
We start by estimating $\dist(I,J)$ as follows:
\begin{equation}\label{eq:estimate for dist(I,J)}
\begin{split}
  \dist(I,J)&\leq\frac{1}{2}(m-1)\ell(I)+\dist(mI,mJ)+\frac{1}{2}(m-1)\ell(J)\\
  &\leq\frac{1}{2}(m-1)\ell(I)+\frac{1}{2}D(I,J)+\frac{1}{2}(m-1)\ell(J)\\
  &=\frac{1}{2}(m-1)\ell(I)+\frac{1}{2}\ell(I)+\frac{1}{2}\dist(I,J)+\frac{1}{2}\ell(J)+\frac{1}{2}(m-1)\ell(J)\\
  &=\frac{m}{2}\ell(I)+\frac{1}{2}\dist(I,J)+\frac{m}{2}\ell(J).
\end{split}
\end{equation}
Hence, (\ref{eq:estimate for dist(I,J)}) implies 
\begin{equation}\label{eq:final estimate for dist(I,J)}
  \dist(I,J)\leq m\ell(I)+m\ell(J).
\end{equation}
Applying (\ref{eq:final estimate for dist(I,J)}) we have
\begin{equation*}
  \frac{D(I,J)}{\ell(J)}=\frac{\ell(I)+\dist(I,J)+\ell(J)}{\ell(J)}\leq(m+1)\frac{\ell(I)+\ell(J)}{\ell(J)}\leq2(m+1).
\end{equation*}
\end{proof}
\end{lemma}

Using Lemma \ref{lem:IveeJ} we can organize the sum $\sigma_{\operatorname{between}}$ in a similar way as the sum $\sigma_{\operatorname{far}}$ 
\begin{equation*}
  \sigma_{\operatorname{between}}=\sum_{j=1}^{\infty}\sum_{i=j}^{\infty}\sum_K\sum_{\substack{I,J:\ell(I)\leq\ell(J)\\\dist(I,J)>\ell(J)(\ell(I)/\ell(J))^\theta\\ \dist(mI,mJ)\leq\frac{1}{2}D(I,J)\\ I\vee J=K\\\ell(I)=2^{-i}\ell(K),\ell(J)=2^{-j}\ell(K)}}a(I,J),
\end{equation*}
where
\begin{equation}\label{eq.1: between cubes}
  a(I,J):=\pair{g}{\psi_J}\pair{\psi_J}{T\psi_I}\pair{\psi_I}{f}
\end{equation}
satisfies, by (\ref{eq3:sigma between and sigma in}), the estimate
\begin{equation}\label{eq.2: between cubes}
\begin{split}
  \abs{a(I,J)}&\lesssim\abs{\pair{g}{\psi_J}}(\Norm{K}{CZ_s}+\Norm{T}{L^2\to L^2})\Big(\frac{|I|}{|J|}\Big)^{1/2}\Big(\frac{\ell(I)}{\ell(J)}\Big)^{s-\epsilon}\abs{\pair{\psi_I}{f}}\\
  &=\abs{\pair{g}{\psi_J}}(\Norm{K}{CZ_s}+\Norm{T}{L^2\to L^2})\frac{\sqrt{|I||J|}}{|K|}\frac{|K|}{|J|}\Big(\frac{\ell(I)}{\ell(J)}\Big)^{s-\epsilon}\abs{\pair{\psi_I}{f}}.
\end{split}
\end{equation}

By combining Lemmata \ref{lem:IveeJ} and \ref{lem:an estimate for D(I,J)} we can estimate $|K|/|J|\Big(\ell(I)/\ell(J)\Big)^{s-\epsilon}$ of (\ref{eq.1: between cubes}) as follows:
\begin{equation}\label{eq.3: between cubes}
\begin{split}
  \frac{|K|}{|J|}\Big(\frac{\ell(I)}{\ell(J)}\Big)^{s-\epsilon}&=\frac{\ell(K)^d}{\ell(J)^d}\Big(\frac{\ell(I)}{\ell(J)}\Big)^{s-\epsilon}\\
  &\lesssim\frac{\ell(K)^d}{\Big({\ell(K)\Big(\frac{\ell(I)}{\ell(K)}\Big)^\theta}\Big)^{d+s-\epsilon}}\ell(I)^{s-\epsilon}\\
  &=\frac{\ell(K)^d}{\ell(K)^{(1-\theta)(d+s-\epsilon)}}\frac{\ell(I)^{s-\epsilon}}{\ell(I)^{\theta(d+s-\epsilon)}}\\
  &=\Big(\frac{\ell(I)}{\ell(K)}\Big)^{s-\epsilon-\theta(d+s-\epsilon)}\\
  &\le\Big(\frac{\ell(I)}{\ell(K)}\Big)^{s-\epsilon-\theta(d+s)}\\
  &=\Big(\frac{\ell(I)}{\ell(K)}\Big)^{s-2\epsilon},
\end{split}
\end{equation}
where we choose $\theta=\frac{\epsilon}{d+s}$ for any given $\epsilon>0$. Summarizing, from (\ref{eq.2: between cubes}) and (\ref{eq.3: between cubes}) we have
\begin{equation*}
\begin{split}
  \sigma_{\operatorname{between}}&=c\sum_{j=1}^{\infty}\sum_{i=j}^{\infty}\sum_K(\Norm{K}{CZ_s}+\Norm{T}{L^2\to L^2})2^{-i(s-2\epsilon)}\pair{g}{A_K^{ij} f}\\
  &=c(\Norm{K}{CZ_s}+\Norm{T}{L^2\to L^2})\sum_{j=1}^{\infty}\sum_{i=j}^{\infty}2^{-i(s-2\epsilon)}\pair{g}{S^{ij} f},
\end{split}
\end{equation*}
where $A^{ij}_K$ is an averaging operator and  $S^{ij}$ is a good wavelet shift with parameters $(i,j)$.

\subsection{Contained cubes, $\sigma_{\operatorname{in}}$}

Let $M>m$. When $I\subsetneq J$, the argument is the same as in the case of the sum $\sigma_{\operatorname{between}}$ but further apart from the corresponding estimate in \cite[Section~3.2]{Hyt:dyadic op.}. Hence, by combining equations (\ref{eq1:sigma between and sigma in}) and (\ref{eq2:sigma between and sigma in}), Lemmata \ref{eq:first term} and \ref{eq:second term}, estimate (\ref{eq3:sigma between and sigma in}) we can organize
\begin{equation*}
  \sigma_{\operatorname{in}}=\sum_{j=1}^{\infty}\sum_{i=j}^{\infty}\sum_K\sum_{\substack{I,J:I\subsetneq J\\I\vee J=K\\\ell(I)=2^{-i}\ell(K),\ell(J)=2^{-j}\ell(K)}}a(I,J),
\end{equation*}
where $a(I,J)$ is defined in (\ref{eq.1: between cubes}) and satisfies the estimate (\ref{eq.2: between cubes}).

We observe that for the contained cubes $\sigma_{\operatorname{in}}$, we have from Lemma \ref{lem:IveeJ} the bound $\ell(K)\Big(\frac{\ell(I)}{\ell(K)}\Big)^\theta\lesssim D(I,J).$ Also, from the definition of the contained cubes we have $D(I,J)\lesssim\ell(J)$, which is the same as the conclusion of Lemma \ref{lem:an estimate for D(I,J)} in the case of $\sigma_{\operatorname{between}}$. Thus, we have all the same auxiliary estimates as in $\sigma_{\operatorname{between}}$ and the same conclusion
\begin{equation*}
\begin{split}
  \sigma_{\operatorname{in}}&=c\sum_{j=1}^{\infty}\sum_{i=j}^{\infty}\sum_K(\Norm{K}{CZ_s}+\Norm{T}{L^2\to L^2})2^{-i(s-2\epsilon)}\pair{g}{A_K^{ij} f}\\
  &=c(\Norm{K}{CZ_s}+\Norm{T}{L^2\to L^2})\sum_{j=1}^{\infty}\sum_{i=j}^{\infty}2^{-i(s-2\epsilon)}\pair{g}{S^{ij} f},
\end{split}
\end{equation*}
where $A^{ij}_K$ is an averaging operator and  $S^{ij}$ is a good wavelet shift with parameters $(i,j)$.

\subsection{Near-by cubes, $\sigma_{=}$ and $\sigma_{\operatorname{near}}$}

We are left to deal with the sums $\sigma_{=}$ of equal cubes $I=J$, as well as $\sigma_{\operatorname{near}}$ of disjoint near-by cubes with $\dist(I,J)\leq\ell(J)(\ell(I)/\ell(J))^\theta$. Since $I$ is good, this necessarily implies that $\ell(I)>2^{-r}\ell(J)$. Then, for a given $J$, there are only boundedly many related $I$ in this sum. Note that in contrast to \cite[Section 3.3]{Hyt:dyadic op.} we compute both sums using good wavelet shifts of type $(i,i)$ and $(i,j)$.

\begin{lemma}\label{estimate for all cubes}
For all $I,J\in\mathscr{D}$, we have
\begin{equation*}
  \abs{\pair{\psi_J}{T\psi_I}}\leq\Norm{T}{L^2\to L^2}.
\end{equation*}
\end{lemma}

\begin{proof}
Using the $L^2$-boundedness of $T$, we estimate simply
\begin{equation*}
  \abs{\pair{\psi_J}{T\psi_I}}\leq\Norm{\psi_J}{2}\Norm{T}{L^2\to L^2}\Norm{\psi_I}{2}=\Norm{T}{L^2\to L^2}.
\end{equation*}
\end{proof}

Using this lemma and applying Lemma \ref{lem:IveeJ} for the good $I=J\in\mathscr{D}$ and a cube $J'\in\mathscr{D}$ adjacent to $I$ (i.e., $\ell(J')=\ell(I)$ and $\dist(I,J')=0$), we have that $K:=I\vee J'$ satisfies $\ell(K)\lesssim\ell(I)$ and $mI\subset K$. Moreover, from Lemma \ref{estimate for all cubes}, we have
\begin{equation*}
  \abs{\pair{\psi_I}{T\psi_I}}\leq\Norm{T}{L^2\to L^2}=\frac{\sqrt{|I||J|}}{|I|}\Norm{T}{L^2\to L^2}\lesssim\frac{\sqrt{|I||J|}}{|K|}\Norm{T}{L^2\to L^2}.
\end{equation*}
Thus, we can organize the sum $\sigma_{=}$ as follows
\begin{equation*}
\begin{split}
  \sigma_{\operatorname{=}}&=\sum_{i=1}^{c}\sum_K\sum_{\substack{I:mI\subset K\\\ell(I)=2^{-i}\ell(K)\\ }}\pair{g}{\psi_I}\pair{\psi_I}{T\psi_I}\pair{\psi_I}{f}\\
  &=c\sum_{i=1}^{c}\sum_K\Norm{T}{L^2\to L^2}\pair{g}{A_K^{ii}f}\\
  &=c\Norm{T}{L^2\to L^2}\sum_{i=1}^{c}\pair{g}{S^{ii}f},
\end{split}
\end{equation*}
where $A^{ii}_K$ is an averaging operator and  $S^{ii}$ is a good wavelet shift with parameters $(i,i)$.

For $I$ and $J$ participating in $\sigma_{\operatorname{near}}$, we conclude from Lemma \ref{lem:IveeJ} that $K:=I\vee J$ satisfies $\ell(K)\lesssim\ell(I)$. Also, from Lemma \ref{estimate for all cubes}, we have
\begin{equation*}
  \abs{\pair{\psi_J}{T\psi_I}}\leq\Norm{T}{L^2\to L^2}\leq\frac{\sqrt{|I||J|}}{|I|}\Norm{T}{L^2\to L^2}\lesssim\frac{\sqrt{|I||J|}}{|K|}\Norm{T}{L^2\to L^2}.
\end{equation*}
Hence, we may organize
\begin{equation*}
\begin{split}
  \sigma_{\operatorname{near}}&=\sum_{i=1}^{c}\sum_{j=1}^i \sum_K\sum_{\substack{I,J:\dist(I,J)\leq\ell(J)(\ell(I)/\ell(J))^\theta \\ I\cap J=\varnothing,I\vee J=K\\\ell(I)=2^{-i}\ell(K),\ell(J)=2^{-j}\ell(K)}}\pair{g}{\psi_J}\pair{\psi_J}{T\psi_I}\pair{\psi_I}{f}\\
  &=c\sum_{i=1}^{c}\sum_{j=1}^i \sum_K\Norm{T}{L^2\to L^2}\pair{g}{A_K^{ij}f}\\
  &=c\Norm{T}{L^2\to L^2}\sum_{i=1}^{c}\sum_{j=1}^i\pair{g}{S^{ij}f},
\end{split}
\end{equation*}
where $A^{ij}_K$ is an averaging operator and  $S^{ij}$ is a good wavelet shift with parameters $(i,j)$.

Summarizing, we have
\begin{equation*}
  \sigma_{=}+\sigma_{\operatorname{near}}=c\Norm{T}{L^2\to L^2}\sum_{j=1}^{c}\sum_{i=j}^{c} \pair{g}{S^{ij}f},
\end{equation*}
where $S^{ij}$ is a good wavelet shift of type $(i,j)$.

\subsection{Synthesis}

We have checked that
\begin{equation*}
\begin{split}
  &\sum_{\ell(I)\leq\ell(J)} 1_{\operatorname{good}}(I)\pair{g}{\psi_J}\pair{\psi_J}{T\psi_I}\pair{\psi_I}{f}\\
  &=c(\Norm{K}{CZ_s}+\Norm{T}{L^2\to L^2})\Big(\sum_{1\leq j\leq i<\infty}(2^{-i(s-\epsilon)}+2^{-i(s-2\epsilon)})\pair{g}{S^{ij}f}\Big),
\end{split}
\end{equation*}
where $S^{ij}$ is a good wavelet shift of type $(i,j)$.

By symmetry (just observing that the cubes of equal size contributed precisely to the presence of the shifts of type $(i,i)$, and that the dual of a shift of type $(i,j)$ is a shift of type $(j,i)$), it follows that
\begin{equation*}
\begin{split}
  &\sum_{\ell(I)>\ell(J)} 1_{\operatorname{good}}(J)\pair{g}{\psi_J}\pair{\psi_J}{T\psi_I}\pair{\psi_I}{f}\\
  &=c(\Norm{K}{CZ_s}+\Norm{T}{L^2\to L^2})\Big(\sum_{1\leq i<j<\infty}(2^{-j(s-\epsilon)}+2^{-j(s-2\epsilon)})\pair{g}{S^{ij}f}\Big)
\end{split}
\end{equation*}
so that altogether
\begin{equation}\label{final sum}
\begin{split}
  &\sum_{I,J} 1_{\operatorname{good}}(\min\{I,J\}) \pair{g}{\psi_J}\pair{\psi_J}{T\psi_I}\pair{\psi_I}{f}\\
  &=c(\Norm{K}{CZ_s}+\Norm{T}{L^2\to L^2})\Big(\sum_{i,j=1}^{\infty}(2^{-\max(i,j)(s-\epsilon)}+2^{-\max(i,j)(s-2\epsilon)})\pair{g}{S^{ij}f}\Big).
\end{split}
\end{equation}
The coefficient in (\ref{final sum}) is dominated by $2^{-\max(i,j)(s-\epsilon')}$, where $\epsilon'=2\epsilon$ for any given $\epsilon>0$. This completes the proof of Theorem \ref{thm:formula}.

\end{document}